\documentclass[aop,seceqn,MSNbibl,citesort,dvips]{arximspdf}
\usepackage{mathbh}
\usepackage{graphicx}

\doi{10.1214/11-AOP689}
\volume{40}
\issue{6}
\pubyear{2012}
\firstpage{2667}
\lastpage{2689}

\makeatletter

\newtheorem{theorem}{Theorem}
\newtheorem{lemma}{Lemma}[section]
\newtheorem{proposition}[lemma]{Proposition}

\newproclaim{remark}[lemma]{Remark}

\newcommand{\arcsinh}{\mathop{\operatorname{arcsinh}}}
\newcommand{\R}{\mathbb{R}}
\newcommand{\Z}{\mathbb{Z}}

\makeatother

\begin{document}
\begin{frontmatter}

\title{Smirnov's fermionic observable away from~criticality\thanksref{T1}}
\runtitle{Smirnov's observable away from criticality}

\thankstext{T1}{Supported by the ANR Grant BLAN06-3-134462, the
EU Marie-Curie RTN CODY, the ERC AG CONFRA, as well as by the Swiss FNS.}

\begin{aug}
\author[A]{\fnms{V.} \snm{Beffara}\corref{}\ead[label=e1]{vbeffara@ens-lyon.fr}}
\and
\author[B]{\fnms{H.} \snm{Duminil-Copin}\ead[label=e2]{hugo.duminil@unige.ch}}
\runauthor{V. Beffara and H. Duminil-Copin}
\affiliation{\'Ecole Normale Sup\'erieure de Lyon and
Universit\'e de Gen\`eve}
\address[A]{Unit\'e de Math\'ematiques Pures et Appliqu\'ees\\
\'Ecole Normale Sup\'erieure de Lyon\\
F-69364 Lyon CEDEX 7\\
France\\
\printead{e1}} 
\address[B]{D\'epartement de Math\'ematiques\\
Universit\'e de Gen\`eve\\
Gen\`eve\\
Switzerland\\
\printead{e2}}
\end{aug}

\received{\smonth{12} \syear{2010}}
\revised{\smonth{7} \syear{2011}}

%
\begin{abstract}
In a recent and celebrated article, Smirnov [\textit{Ann. of Math.} (2)
\textbf{172} (2010) 1435--1467] defines an \textit{observable} for the
self-dual random-cluster model with cluster weight $q=2$ on the square
lattice $\Z^2$, and uses it to obtain conformal invariance in the
scaling limit. We study this observable away from the self-dual point.
From this, we obtain a new derivation of the fact that the self-dual
and critical points coincide, which implies that the critical inverse
temperature of the Ising model equals $\frac12\log(1+\sqrt2)$.
Moreover, we relate the correlation length of the model to the large
deviation behavior of a certain massive random walk (thus confirming an
observation by Messikh [The surface tension near criticality of the
2d-Ising model (2006) Preprint]), which allows us to compute it
explicitly.
\end{abstract}

%

\end{frontmatter}

\section*{Introduction}

The Ising model was introduced by Lenz \cite{Lenz} as a model for
ferromagnetism. His student, Ising, proved in his Ph.D. thesis \cite{Ising}
that the model does not exhibit any phase transition in one dimension.
On the square lattice $\mathbb{L} = (\mathbb{Z}^2, \mathbb{E})$, the
Ising model is the first model where phase transition and non-mean-field
behavior have been established (this was done by Peierls
\cite{Peierls}).

An Ising configuration is a random assignment of spins $\{-1,1\}$ on
$\Z^2$ such that the probability of a configuration $\sigma$ is
proportional to $\exp[\beta\sum_{a\sim b}\sigma(a)\sigma(b)]$, where
$\beta$ is the inverse temperature of the model and $a\sim b$ means that
$(a,b)$ is an edge of the lattice, that is, $(a,b) \in\mathbb{E}$.
Kramers and Wannier \cite{KramersWannier} identified (without proof) the critical temperature
where a phase transition occurs, separating an ordered from a disordered
phase, using planar duality. In 1944, Kaufman and Onsager
\cite{KaufmanOnsager} computed the free energy of the model, paving the
way to an analytic derivation of its critical temperature. In 1987,
Aizenman, Barsky and Fern\'andez \cite{AizenmanBarskyFernandez} found a
computation of the critical temperature based on differential
inequalities. Both strategies are quite involved, and the first goal of
this paper is to propose a new method, relying only on what we will call
Smirnov's observable:
%
%
\begin{theorem}
\label{maintheorem}
The critical inverse temperature of the Ising model on the square
lattice $\Z^2$ is equal to
\[
\beta_c=\tfrac{1}{2}\ln\bigl( 1+\sqrt{2}\bigr).
\]
\end{theorem}

Beyond the determination of the critical inverse temperature, physicists
and mathematicians were interested in estimates for the correlation
between two spins, $\mathbb{E}_{\beta}[\sigma(a)\sigma(b)]$ (where
$\mathbb{E}_{\beta}$ denotes the Ising measure). McCoy and Wu
\cite{McCoyWu} derived a closed formula for the two-point function, and
an asymptotic analysis shows that it decays exponentially quickly when
$\beta<\beta_c$. In addition to this, it was noticed by Messikh
\cite{Messikh} that the rate of decay is connected to large deviations
estimates for the simple random walk. In this article, we present a
direct derivation of this link, which provides a quick proof of the
following theorem:
%
%
\begin{theorem}
\label{maintheorem2}
Let $\beta<\beta_c$ and let $\mathbb{E}_{\beta}$ denote the (unique)
infinite-volume Ising measure at inverse temperature $\beta$; fix
$a=(a_1,a_2)\in\mathbb{L}$. Then
\[
\lim_{n\rightarrow\infty} -\frac{1}{n} \ln(
\mathbb{E}_{\beta} [\sigma(0)\sigma(n a)] ) = a_1 \arcsinh
sa_1 + a_2 \arcsinh sa_2,
\]
where $s$ solves the equation
\[
\sqrt{1+(sa_1)^2}+\sqrt{1+(sa_2)^2}=\sinh2\beta+\sinh^{-1}
2\beta.
\]
\end{theorem}

Instead of working with the Ising model, we rather deal with its
\textit{random-cluster representation} (known as the
\textit{random-cluster model with cluster weight $q=2$}). It is well
known \cite{FortuinKasteleyn} that one can couple this model with the
Ising model (see, e.g., \cite{Grimmett} for a comprehensive study of
random-cluster models) in such a way that the spin correlations of the
Ising model get rephrased as cluster connectivity properties of their
random-cluster representations, which allows for the use of geometric
techniques. For instance, the determination of $\beta_c$ is equivalent
to the determination of the critical point $p_c$ for the random-cluster
model.

The understanding of the two-dimensional random-cluster model with
$q=2$ has recently progressed greatly \cite{Smirnov,Duminil}, thanks to
the use of the so-called \textit{fermionic observable} introduced by
Smirnov \cite{Smirnov}, which was instrumental in the proof of
conformal invariance. This observable is defined on the edges of a
finite domain with Dobrushin boundary conditions (mixed free and wired;
see Section~\ref{secbasic} for a formal definition), and it is discrete
holomorphic at the self-dual point
$p_{\mathrm{sd}}=\sqrt2/(1+\sqrt2)$.\vadjust{\goodbreak}

The idea of our argument is the following. Below the self-dual point,
the observable can still be defined, but discrete holomorphicity fails,
and the observable decays exponentially quickly in the distance to the
wired boundary. Along the free boundary, the modulus of the observable
can be written exactly as a connection probability, so in the $p<p_{\mathrm{sd}}$
regime the two-point function is exponentially small as well, and that
implies that the system is then in the subcritical regime, thus
providing the lower bound $p_c \geq p_{\mathrm{sd}}$ on the critical parameter.
Theorem~\ref{maintheorem} then follows from duality.

In fact, the rate of exponential decay (and therefore
Theorem~\ref{maintheorem2}) can be derived by comparing the observable
to the Green function of a massive random walk
(Proposition~\ref{messikh}); the key ingredient is the observation that
the observable is massive harmonic in the bulk for $p<p_{\mathrm{sd}}$. The
correspondence between the two-point function of the Ising model and
that of the massive random walk was previously noticed by Messikh
\cite{Messikh}.

In Section~\ref{secbasic}, we remind the reader of a few classic
features of the random-cluster model. In Section~\ref{secdefinition},
we define Smirnov's observable away from criticality and gather some of
its important properties---for instance, the fact that the observable
on a graph is related to connection properties for sites on the
boundary. In Section~\ref{secproofthm1}, we derive
Theorem~\ref{maintheorem} by showing that the observable decays
exponentially fast. Section~\ref{secproofthm2} is devoted to a
refinement of estimates on the observable, which leads to the proof of
Theorem~\ref{maintheorem2}.

\section{Basic features of the model}
\label{secbasic}

The Ising model on the square lattice admits a classical representation
through the so-called \textit{random-cluster model with $q=2$}. This model
can be studied using geometric arguments which are classic in the theory
of lattice models. We list here a few basic features of random-cluster
models; a more exhaustive treatment (together with the proofs of all our
statements) can be found in Grimmett's monograph \cite{Grimmett}.
Readers familiar with the subject can skip directly to the next section.

\subsection*{Definition of the random-cluster model}

The random-cluster measure can be defined on any graph. However, we will
restrict ourselves to the square lattice, denoted by $\mathbb{L} =
(\mathbb{Z}^2,\mathbb{E})$ with $\mathbb{Z}^2$ denoting the set of
\textit{sites} and $\mathbb{E}$ the set of \textit{bonds}. In this paper,
$G$ will always denote a connected subgraph of $\mathbb{L}$, that is,
a subset of vertices of $\mathbb{Z}^2$ together with all the bonds
between them. We denote by $\partial G$ the (inner) boundary of $G$,
that is, the set of sites of $G$ linked by a bond to a site of
$\mathbb{Z}^2\setminus G$.

A \textit{configuration} $\omega$ on $G$ is a random subgraph of $G$,
having the same sites and a subset of its bonds. We will call the bonds
belonging to $\omega$ \textit{open}, the others \textit{closed}. Two sites
$a$ and $b$ are said to be \textit{connected} (denoted by
$a\leftrightarrow b$), if there is an \textit{open path}---a path
composed of open bonds only---connecting them. The (maximal) connected
components will be called \textit{clusters}. More generally, we extend
this definition and notation to sets in a straightforward way.\vadjust{\goodbreak}

A \textit{boundary condition} $\xi$ is a partition of $\partial G$. We
denote by $\omega\cup\xi$ the graph obtained from the configuration
$\omega$ by identifying (or \textit{wiring}) the vertices in $\xi$ that
belong to the same class of $\xi$. A boundary condition encodes the way
in which sites are connected outside of $G$. Alternatively, one can see
it as a collection of \textit{abstract bonds} connecting the vertices in
each of the classes to each other. We still denote by $\omega\cup\xi$
the graph obtained by adding the new bonds in $\xi$ to the configuration
$\omega$, since this will not lead to confusion. Let $o(\omega)$
[resp., $c(\omega)$] denote the number of open (resp., closed) bonds of
$\omega$
and $k(\omega,\xi)$ the number of connected components of
$\omega\cup\xi$. The probability\vspace*{1pt} measure $\phi^{\xi}_{p,q,G}$ of the
random-cluster model on a \textit{finite} subgraph $G$ with parameters
$p\in[0,1]$ and $q\in(0,\infty)$ and boundary condition $\xi$ is defined
by
%
%
\begin{equation}
\label{probconf}
\phi_{p,q,G}^{\xi} (\{\omega\}) := \frac
{p^{o(\omega)}(1-p)^{c(\omega)}q^{k(\omega,\xi)}} {Z_{p,q,G}^{\xi}}
\end{equation}
for any subgraph $\omega$ of $G$, where $Z_{p,q,G}^{\xi}$ is a
normalizing constant known as the \textit{partition function}. When there
is no possible confusion, we will drop the reference to parameters in
the notation.

\subsection*{The domain Markov property}

One can encode, using an appropriate boundary condition $\xi$, the
influence of the configuration outside a sub-graph on the measure within
it. Consider a graph $G=(V,E)$ and a random-cluster measure
$\phi^{\psi}_{p,q,G}$ on it. For $F\subset E$, consider $G'$ with $F$
as the set of edges and the endpoints of it as the set of sites.
Then,\vspace*{-1pt}
the restriction to $G'$ of $\phi^{\psi}_{p,q,G}$ conditioned to
match some configuration $\omega$ outside $G'$ is exactly
$\phi_{p,q,G'}^{\xi}$, where $\xi$ describes\vspace*{1pt} the connections
inherited from $\omega\cup\psi$ (two sites are wired if they are
connected by a path in $\omega\cup\psi$ outside $G'$; see (4.13) in
\cite{Grimmett}). This property is the direct analog of the DLR
conditions for spin systems.

\subsection*{Comparison of boundary conditions when $q\geq1$}

An event is called \textit{increasing} if it is preserved by addition of
open edges. When $q\geq1$, the model is \textit{positively correlated}
(see (4.14) in \cite{Grimmett}), which has the following consequence:
for any boundary conditions $\psi\leq\xi$ (meaning that $\psi$ is finer
than $\xi$, or in other words, that there are fewer connections in
$\psi$ than in $\xi$), we have
%
%
\begin{equation}
\label{comparisonbetweenboundaryconditions}
\phi^{\psi}_{p,q,G}(A)\leq\phi^{\xi}_{p,q,G}(A)
\end{equation}
for any increasing event $A$. This last property, combined with the
Domain Markov property, provides a powerful tool in order to study how
events decorrelate.

\subsection*{Examples of boundary conditions: free, wired and Dobrushin}

Two boundary conditions play a special role in the study of
random-cluster models: the \textit{wired} boundary condition, denoted by
$\phi_{p,q,G}^1$, is specified by the fact that all the vertices\vadjust{\goodbreak} on the
boundary are pairwise connected; the \textit{free} boundary condition,
denoted by $\phi_{p,q,G}^0$, is specified by the absence of
wirings between boundary sites. These boundary conditions are extremal
for stochastic ordering, since any other boundary condition is smaller
(resp.,
greater) than the wired (resp., free) boundary condition.

Another example of boundary condition will be very useful in this paper.
The following definition is deliberately not as general as would be
possible, in order to limit the introduction of notation. Let $G$ be a
finite subgraph of $\mathbb L$; assume that its boundary is a
self-avoiding polygon in $\mathbb L$, and let $a$ and $b$ be two sites
of $\partial G$. The triple $(G,a,b)$ is called a \textit{Dobrushin
domain}. Orienting its boundary counterclockwise defines two oriented
boundary arcs $ab$ and $ba$; the \textit{Dobrushin boundary condition} is
defined to be free on $ab$ (there are no wirings between boundary sites)
and wired on $ba$ (all the boundary sites are pairwise connected). We
will refer to those arcs as the \textit{free arc} and the \textit{wired
arc}, respectively. The measure associated to this boundary condition
will be denoted by $\phi_{p,q,G}^{a,b}$ or simply $\phi_{G}^{a,b}$.

\subsection*{Planar duality for Dobrushin domains}

One can associate to any random-cluster measure with parameters $p$ and
$q$ on a Dobrushin domain $(G,a,b)$ a dual measure. First, define the
\textit{dual graph} $G^*$ as follows: place a site in the center of every
face of $G$ and every face of $\mathbb{L}$ adjacent to the free arc; see
Figure~\ref{figmediallattice}. Bonds of the dual graph correspond to
bonds of the primal graph and link nearest neighbors. Construct a bond
%
%
\begin{figure}

\includegraphics{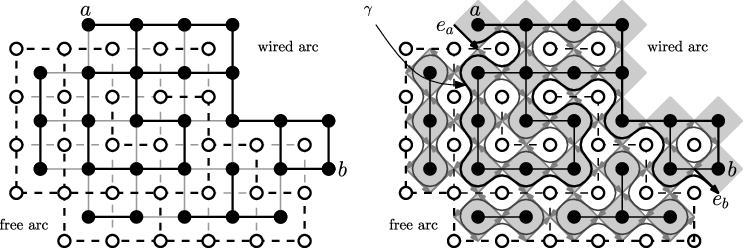}

\caption{Left: a graph $G$ with its dual $G^*$. The black
(resp., white) sites are the sites of $G$ (resp.,~$G^*$). The open
bonds of $G$ (resp., $G^*$) are represented by solid (resp., dashed)
black bonds. Right: construction of the medial lattice and
the loop representation: the loops are interfaces between primal and
dual clusters.}
\label{figmediallattice}
\end{figure}
model on $G^*$ by declaring any bond of the dual graph to be open
(resp., closed) if the corresponding bond of the primal lattice is
closed (resp., open) for the initial random-cluster model. The new model
on the dual graph is then a random-cluster measure with parameters $p^*
= p^*(p,q)$ and $q^*=q$ satisfying
\[
p^*(p,q) :=
\frac{(1-p)q}{(1-p)q+p}\quad \mbox{or equivalently}\quad
\frac{p^*p}{(1-p^*)(1-p)}=q
\]
with wired boundary condition on the dual
arc adjacent to $ab$, and free boundary condition on the dual arc
adjacent to $ba$. In particular, it is again a random-cluster model with
Dobrushin boundary condition. This relation is known as \textit{planar
duality}. It is then natural to define the self-dual point
$p_{\mathrm{sd}}=p_{\mathrm{sd}}(q)$ by solving the equation $p^*(p_{\mathrm{sd}},q)=p_{\mathrm{sd}}$, which
gives
\[
p_{\mathrm{sd}}(q) := \frac{\sqrt{q}} {1+\sqrt{q}}.
\]

This notion of duality has a natural counterpart, with the same formal
definition, for free boundary conditions: the dual model is then a
random-cluster model with parameters $p^*$ and $q$, with wired boundary
condition.

\subsection*{Infinite-volume measures and the critical point}

The domain Markov property and comparison between boundary conditions
allow us to define infinite-volume measures. Indeed, one can consider a
sequence of measures on boxes of increasing sizes with free boundary
conditions. This sequence is increasing in the sense of stochastic
domination, which implies that it converges weakly to a limiting
measure, called the random-cluster measure on $\mathbb{L}$ with free
boundary conditions (denoted by $\phi_{p,q}^0$). This classic
construction can be performed with many other sequences of measures,
defining a priori different infinite-volume measures on
$\mathbb{L}$. For instance, one can define the random-cluster measure
$\phi_{p,q}^1$ with wired boundary conditions, by considering the
decreasing sequence of random-cluster measures on finite boxes with
wired boundary condition.

For our purpose, the following example of infinite-volume measure will
be important: we define a measure on the strip $\mathcal{S}_{\ell} =
\mathbb{Z}\times[ 0,\ell]$. The sequence of measures $(\phi
_{[-m,m]\times[0,{\ell}]} ^{(m,0),(-m,0)}) _{m\geq0}$ is
\textit{increasing}, in the sense that for any cylindrical increasing
event $A$ defined in the strip, the sequence $(\phi
_{[-m,m]\times[0,{\ell}]} ^{(m,0),(-m,0)} (A))$ is well\vspace*{1pt}
defined for $m$ large enough and is nondecreasing. This implies that
the sequence of measures converges weakly as $m$ goes to infinity. The
limit is called the random-cluster measure on the infinite strip with
free boundary conditions on the top and wired boundary condition on the
bottom, and we will denote it by $\phi_{\mathcal{S}_{\ell}}
^{\infty,-\infty}$.

When defining such measures in infinite volume by thermodynamical
limits, it is natural to ask whether the limit depends on the choice of
domains and boundary conditions used to build it; in the case of the
random-cluster model, a more specific version of the question is whether
taking free or wired boundary conditions affects the limit---these two
being extremal, if the limits match, this implies uniqueness of the
infinite-volume limit for all boundary conditions. It can be shown that
for fixed $q\geq1$, uniqueness can fail only on a countable set
$\mathcal{D}_q$ of values of $p$; see Theorem (4.60) of \cite{Grimmett}.
From that (or rather from the weaker statement that the set of values of
$p$ at which uniqueness holds is everywhere dense\vadjust{\goodbreak} in $[0,1]$), and from
the fact that measures for larger values of $p$ dominate those for
smaller values, it is not difficult to show that there exists a
\textit{critical point} $p_c$ such that for any infinite-volume measure
with $p<p_c$ (resp., $p>p_c$), there is almost surely no infinite
component of connected sites (resp., at least one infinite
component). Moreover, it is also known that the infinite-volume measure
is unique when $p<p_{\mathrm{sd}}$.
%
%
\begin{remark}
Physically, it is natural to conjecture that the critical point
satisfies $p_c=p_{\mathrm{sd}}$. Indeed, if one assumes $p_c\neq p_{\mathrm{sd}}$, there
should be a phase transition due to the change of behavior in the
primal model at $p_c$ and a second (different) phase transition due to
the change of behavior in the dual model at $p_c^*$. This is unlikely
to happen---in fact, constructing a natural-looking model exhibiting
two phase transitions is not so easy; but the equality of $p_c$ and
$p_{\mathrm{sd}}$ is only known to hold in a few specific cases.

In the case of the random-cluster model on the square lattice, the
authors proved recently \cite{BeffaraDuminil} that indeed
$p_c(q)=p_{\mathrm{sd}}(q)$ for all $q\geq1$ (therefore determining the
critical temperature for all $q$-state Potts models on
$\mathbb{L}$). The argument does not use Smirnov's observable, but it
is quite a bit longer than the one we present here, is not as
self-contained (mostly because it depends on recent sharp-threshold
results by Graham and Grimmett \cite{GrahamGrimmett,GrahamGrimmett2})
and it provides less information on the subcritical phase.
\end{remark}

\subsection*{Coupling with the Ising model}

The random-cluster model on $G$ with parameter $q=2$ is of particular
interest since it can be coupled with the Ising model; consider a
configuration $\omega$ sampled with probability $\phi_{p,2,G}^0$ and
assign independently a spin $+1$ or $-1$ to every cluster with
probability $1/2$. We are now facing a model of spins on sites of $G$. It
can be proved that the law of the configuration corresponds to the Ising
model at temperature $\beta=\beta(p)=-\frac{1}{2}\ln(1-p)$ with free
boundary condition.

We are then equipped with a ``dictionary'' between the properties of the
random-cluster model with $q=2$ and those of the Ising model. One
instance of this relation is given by the useful identity
%
%
\begin{equation}
\label{couplingFKIsing}
\mathbb{E}_{\beta(p),G}^{\mathrm{free}} [\sigma(0)\sigma(a)] =
\phi_{p,2,G}^0 (0\leftrightarrow a),
\end{equation}
where the left-hand term denotes the correlation between sites $0$ and
$a$ for the Ising model at inverse temperature $\beta$ on the graph $G$
with free boundary condition.

The critical inverse temperature $\beta_c$ of the Ising model is
characterized by the fact that the two-point correlation undergoes a
phase transition in its asymptotic behavior: below $\beta_c$, the
correlation goes to 0 when $a$ goes to infinity, while above it, it
stays bounded away from 0. The previous definition readily implies that
$\beta_c = - \frac12 \log(1-p_c(2))$. In order to prove
Theorem~\ref{maintheorem}, it is thus sufficient to determine $p_c(2)$.
Notice that the inverse temperature corresponding to the self-dual point
is given by $\beta(p_{\mathrm{sd}})=\frac{1}{2}\ln(1+\sqrt{2})$ so that what
needs to be proved can be written as $p_c(2)=p_{\mathrm{sd}}(2)$.\vadjust{\goodbreak}

The same reasoning implies that we can compute correlation lengths for
the random-cluster model in order to prove Theorem~\ref{maintheorem2}.

\section{Definition of the observable}
\label{secdefinition}

\textit{From now on, we consider only random-cluster models on the
two-dimensional square lattice with parameter $q=2$} (we drop the
dependency on $q$ in the notation).

\subsection*{The medial lattice and the loop representation}

Let $(G,a,b)$ be a Dobrushin domain. In this paragraph, we aim for the
construction of the loop representation of the random-cluster model,
defined on the so-called medial graph. In order to do that, consider
$G$ together with its dual $G^*$; declare \textit{black} the sites of $G$
and \textit{white} the sites of $G^*$. Replace every site with a colored
diamond, as in Figure~\ref{figmediallattice}. The \textit{medial graph}
$G_{\diamond}=(V_{\diamond},E_{\diamond})$ is defined as follows (see
Figure~\ref{figmediallattice} again): $E_\diamond$ is the set of
diamond sides which belong to both a black and a white diamond;
$V_\diamond$ is the set of all the endpoints of the edges in
$E_\diamond$. We obtain a subgraph of a rotated (and rescaled) version
of the usual square lattice. We give $G_\diamond$ an additional
structure as an oriented graph by orienting its edges clockwise around
white faces.

The random-cluster measure with Dobrushin boundary condition has a
rather convenient representation in this setting. Consider a
configuration $\omega$: it defines clusters in $G$ and dual clusters in
$G^*$. Through every vertex of the medial graph passes either an open
bond of $G$ or a dual open bond of $G^*$. Hence, there is a unique way
to draw Eulerian (i.e., using every edge of $E_{\diamond}$ exactly
once) loops on the medial lattice such that the loops are the
\textit{interfaces} separating primal clusters from dual clusters. Namely,
a loop arriving at a vertex of the medial lattice always makes a $\pm
\pi/2$ turn so as not to cross the open or dual open bond through this
vertex; see Figure~\ref{figmediallattice} yet again.

Besides loops, the configuration will have a single curve joining the
vertices adjacent to $a$ and $b$, which are the only vertices in
$V_{\diamond}$ with three adjacent edges within the domain (the fourth
edge emanating from $a$, resp., $b$, will be denoted by~$e_a$, resp.,
$e_b$). This curve is called the \textit{exploration path}; we will denote
it by $\gamma$. It corresponds to the interface between the cluster
connected to the wired arc and the dual cluster connected to the free
arc.

This gives a bijection between random-cluster configurations on $G$ and
Eulerian loop configurations on $G_{\diamond}$. The probability measure
can be nicely rewritten (using Euler's formula) in terms of the loop
picture
\[
\phi_{G}^{a,b} ({\omega}) = \frac{x(p)^{\#\ \mathrm{open}
\ \mathrm{bonds}} \sqrt{2}{}^{\#\ \mathrm{loops}}} {\tilde{Z}(p,G)}\qquad
\mbox{where } x(p) := \frac p {(1-p) \sqrt{2}},
\]
and
$\tilde{Z}(p,G)$ is a normalizing constant. Notice that $p=p_{\mathrm{sd}}$ if
and only if $x(p)=1$. This bijection is called the \textit{loop
representation} of the random-cluster model. The orientation of the
medial graph gives a natural orientation to the interfaces in the loop
representation.

\subsection*{The edge observable for Dobrushin domains}

Fix a Dobrushin domain $(G,a,b)$. Following \cite{Smirnov}, we now
define an observable $F$ on the edges of its medial graph, that is,
a function $F\dvtx E_{\diamond} \to\mathbb{C}$. Roughly speaking, $F$
is a
modification of the probability that the exploration path passes through
an edge.\looseness=-1

%
%
\begin{figure}

\includegraphics{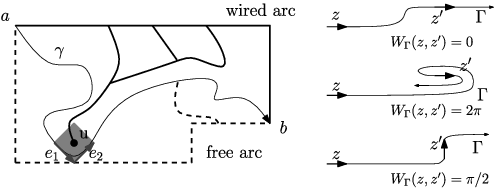}

\caption{Left: a schematic picture of the exploration path
and a boundary point $u$, together with two possible choices $e_1$
and $e_2$ for $e$. If $u$ is connected to the wired arc, the
exploration path must go through $e$. Right: the winding of
a curve. In the first example, the curve did one quarter-turn on the
left and one quarter-turn on the right.}
\label{figboundary}
\end{figure}

First, introduce the following definition: the \textit{winding}
$\mathrm{W}_{\Gamma}(z,z')$ of a curve $\Gamma$ between two edges $z$ and
$z'$ of the medial graph is the total rotation (in radians and oriented
counter-clockwise) that the curve makes from the mid-point of edge $z$
to that of edge $z'$; see Figure~\ref{figboundary}. We define the
observable $F$ for any edge $e\in E_{\diamond}$ as
%
%
\begin{equation}
\label{defF}
F(e) := \phi_{G}^{a,b} \bigl(\mathrm{e}^{({\mathrm i}/{2})
\mathrm{W}_{\gamma}(e,e_b)} \mathbh{1}_{e\in\gamma}\bigr),
\end{equation}
where $\gamma$ is the exploration path.
%
%
\begin{remark}
In \cite{Smirnov}, Smirnov extends the observable to vertices---as
being the sum of $F$ on adjacent edges---in order to study the
critical regime. Properly rescaled, this function converges to a
holomorphic function, which is a key step toward the proof of
conformal invariance; and indeed the exploration curve $\gamma$
converges to the trace of an SLE process as the mesh goes to $0$. Away
from criticality, it is more convenient to work directly with the
observable on edges.
\end{remark}

The following three lemmas present the properties of the observable we
will be using in the proofs of both theorems. They have direct
counterparts in Smirnov's article \cite{Smirnov} (in particular, the
idea of the proof of Lemma~\ref{boundary} can be found in the proof of
Lemma 4.12 of \cite{Smirnov}), and as such they are not completely
new. We still include their proofs here since our goal is to keep the
present paper as self-contained as possible.
%
%
\begin{lemma}
\label{boundary}
Let $u\in G$ be a site on the free arc, and $e$ be a side of the black
diamond associated to $u$ which borders a white diamond of the free
arc; see Figure~\ref{figboundary}. Then
%
%
\begin{equation}
|F(e)|=\phi_{G}^{a,b}(u\leftrightarrow\mbox{wired arc}).
\end{equation}
\end{lemma}
\begin{pf}
Let $u$ be a site of the free arc and recall that the exploration path
is the interface between the open cluster connected to the wired arc
and the dual open cluster connected to the free arc. Since $u$ belongs
to the free arc, $u$ is connected to the wired arc if and only if $e$
is on the exploration path, so that
\[
\phi_{G}^{a,b}(u\leftrightarrow
\mbox{wired arc})=\phi_{G}^{a,b}(e\in\gamma).
\]
The edge $e$ being on
the boundary, the exploration path cannot wind around it, so that the
winding (denoted $\mathrm{W}_1$) of the curve is deterministic (and easy to
write in terms of that of the boundary itself). We deduce from this
remark that
\begin{eqnarray*}
|F(e)|&=&\bigl|\phi_{G}^{a,b} \bigl({\mathrm e}^{({\mathrm i}/{2}) \mathrm{W}_1}
\mathbh{1}_{e\in\gamma}\bigr)\bigr| = \bigl|{\mathrm e}^{({\mathrm i}/{2})
\mathrm{W}_1}\phi_{G}^{a,b}(e\in\gamma)\bigr| \\
&=& \phi_{G}^{a,b}(e\in
\gamma)=\phi_{G}^{a,b}(u\leftrightarrow\mbox{wired arc}).
\end{eqnarray*}
\upqed
\end{pf}

For a random-cluster model, one can use the parameters $p$ or $x$
interchangeably. We introduce a third parameter which will be
convenient: let $\alpha=\alpha(p)\in[0,2\pi)$ be given by the relation
%
%
\begin{equation}\label{eqalphafromp}
{\mathrm e}^{{\mathrm i}\alpha(p)}:= \frac{{\mathrm e}^{{\mathrm i}\pi/4}
+x(p)}{{\mathrm e}^{{\mathrm i}\pi/4}x(p)+ 1}.
\end{equation}
Observe that $\alpha(p)=0$ if and only if $p=p_{\mathrm{sd}}$ and $\alpha(p)>0$
for $p<p_{\mathrm{sd}}$. With this definition:
%
%
\begin{lemma}
\label{relationaroundavertex}
Consider a vertex $v\in V_{\diamond}$ with four adjacent edges in
$E_{\diamond}$. For every $p\in[0,1]$,
%
%
\begin{equation}
\label{relvertex}
F(A)+F(C) ={\mathrm e}^{{\mathrm i}\alpha(p)} [F(B)+F(D)],
\end{equation}
where $A$ and $C$ (resp., $B$ and $D$) are the adjacent edges pointing
toward (resp., away from) $v$, as depicted in
Figure~\ref{figconfiguration}.
\end{lemma}
\begin{pf}
Let $v$ be a vertex of $V_{\diamond}$ with four adjacent edges,
indexed as mention above. Edges $A$ and $C$ play symmetric roles, so
that we can further require the indexation to be in clockwise order
(see one such indexation in Figure~\ref{figconfiguration}). Recall
that any vertex in $V_{\diamond}$ corresponds to a bond of the primal
graph \textit{and} a bond of the dual graph. We consider the involution
$s$ on the space of configurations which switches the state (open or
closed) of the bond of the primal lattice corresponding to $v$.

%
%
\begin{figure}

\includegraphics{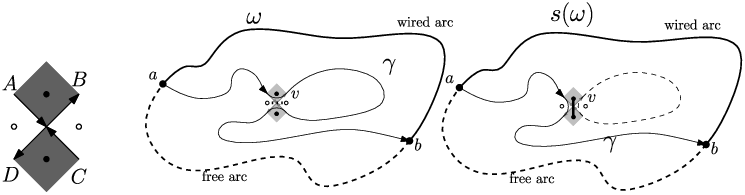}

\caption{Left: indexation of the edges adjacent to
$v$. Right: two associated configurations~$\omega$ and~$s(\omega)$. In
this picture, $v$ corresponds to a vertical bond of the primal
lattice.} \label{figconfiguration}
\end{figure}

Let $e$ be an edge of the medial graph and denote by $e_{\omega} :=
\phi_{G}^{a,b}(\omega) {\mathrm e}^{({\mathrm
i}/{2})\mathrm{W}_{\gamma}(e,e_b)}\times \mathbh{1}_{e\in\gamma}$ the
contribution of $\omega$ to $F(e)$. Since $s$ is an involution, the
following relation holds:
\[
F(e):=\sum_{\omega} e_{\omega}=\frac{1}{2} \sum_{\omega}
\bigl[ e_{\omega}+e_{s(\omega)} \bigr].
\]
In order to prove
(\ref{relvertex}), it suffices to prove the following for any
configuration~$\omega$:
%
%
\begin{equation}
\label{c}
A_{\omega} + A_{s(\omega)} + C_{\omega} + C_{s(\omega)} =
{\mathrm e}^{{\mathrm i}\alpha(p)} \bigl( B_{\omega} + B_{s(\omega)} +
D_{\omega} +
D_{s(\omega)} \bigr).
\end{equation}
When $\gamma(\omega)$ does not go through any of the edges adjacent to
$v$, it is easy to see that neither does $\gamma(s(\omega))$. All the
contributions then vanish and (\ref{c}) trivially holds. Thus we can
assume that $\gamma(\omega)$ passes through at least one edge adjacent
to $v$. The interface follows the orientation of the medial graph, and
thus can enter $v$ through either $A$ or $C$ and leave through $B$ or
$D$. Without loss of generality we assume that it enters first through
the edge $A$ and leaves last through the edge $D$; the other cases are
treated similarly.

Two cases can occur: either the exploration curve, after arriving
through~$A$, leaves through $B$ and then returns a second time through
$C$, leaving through~$D$; or the exploration curve arrives through~$A$
and leaves through~$D$, with $B$ and $C$ belonging to a loop. Since
the involution exchanges the two cases, we can assume that $\omega$
corresponds to the first case. Knowing the term $A_{\omega}$, it is
possible to compute the contributions of $\omega$ and $s(\omega)$ to
all of the edges adjacent to $v$. Indeed:
\begin{itemize}
\item The probability of $s(\omega)$ is equal to $x(p)\sqrt{2}$ times
the probability of $\omega$ (due to the fact that there is one
additional open edge and one additional loop).
\item Windings of the curve can be expressed using the winding at
$A$. For instance, the winding at $B$ in the configuration $\omega$
is equal to the winding at $A$ minus a $\pi/2$ turn.
\end{itemize}
The contributions are given as:
\begin{center}
\tablewidth=315pt
\begin{tabular*}{\tablewidth}{@{\extracolsep{\fill}}lcccc@{}}
\hline
\textbf{Configuration} & $\bolds{A}$ & $\bolds{C}$ & $\bolds{B}$ & $\bolds{D}$\\
\hline
$\omega$ & $A_{\omega}$ & ${\mathrm e}^{{\mathrm i}\pi/2}A_{\omega}$ &
${\mathrm e}^{-{\mathrm i}\pi/4}A_{\omega}$ & ${\mathrm e}^{{\mathrm
i}\pi/4}A_{\omega}$\\
[3pt]
$s(\omega)$ & $x(p)\sqrt{2}A_{\omega}$ & 0 & 0 & ${\mathrm e}^{{\mathrm
i}\pi/4}x(p)\sqrt{2}A_{\omega}$ \\
\hline
\end{tabular*}
\end{center}

\vspace*{6pt}
\noindent Using the identity ${\mathrm e}^{{\mathrm i}\pi/4}+{\mathrm
e}^{-{\mathrm i}\pi/4}=\sqrt{2}$, we deduce (\ref{c}) by summing the
contributions of all the edges around $v$.
\end{pf}

The previous lemma provides us with one linear relation between values
of $F$ for every vertex inside the domain. However, there are
approximately twice as many edges than vertices in $G_{\diamond}$ so
that these relations do not completely determine the value of $F$. The
next lemma is therefore crucial since it decreases the number of
possible values for $F$; roughly speaking, it states that the complex
argument (modulo $\pi$) of $F(e)$ is determined by the orientation of
the edge $e$.
%
%
\begin{lemma}
\label{argument}
$F(e)$ belongs to $\mathbb{R}$ (resp., ${\mathrm e}^{-{\mathrm i}\pi
/4}\mathbb{R}$, ${\mathrm i}\mathbb{R}$ or ${\mathrm e}^{{\mathrm i}\pi
/4}\mathbb{R}$) on edges $e$ pointing in the same direction
as the ending edge $e_b$ (resp., edges pointing in a direction which
forms an angle $\pi/2$, $\pi$ and $3\pi/2$ with $e_b$).
\end{lemma}
\begin{pf}
The winding at an (oriented) edge can only take its value in the set
$\mathrm{W}_0+2\pi\mathbb{Z}$ where $\mathrm{W}_0$ is the winding at $e$ of an arbitrary
possible interface passing through $e$. Therefore, the winding weight
involved in the definition of $F$ is always proportional to ${\mathrm
e}^{{\mathrm i}\mathrm{W}_0/2}$ with a real-valued coefficient, and thus the
complex argument of $F$ is equal to $\mathrm{W}_0/2$ or $\mathrm{W}_0/2+\pi$. Since
$\mathrm{W}_0$ is exactly the angle between the direction of $e$ and that of
$e_b$, we obtain the result.
\end{pf}

\subsection*{The observable in strips}

The definition of $F$ can be extended to the case of the strip. Indeed,
the loop representation extends in this setting; the
$\phi_{\mathcal{S}_{\ell}}^{\infty,-\infty}$-probability of having an
infinite cluster is $0$: for fixed $\ell$, the model is essentially
one dimensional, and it is a simple exercise to prove that it must be
subcritical. Hence, there is a \textit{unique} interface going from
$+\infty$ to $-\infty$, which we call $\gamma$. We define
\[
F(e) :=
\phi_{\mathcal{S}_{\ell}}^{\infty,-\infty} \bigl[ {\mathrm e}^{
({\mathrm i}/{2})\mathrm{W}_{\gamma}(e,-\infty)} \mathbh{1}_{e\in\gamma}
\bigr],
\]
where $\mathrm{W}_{\gamma}(e,-\infty)$ is the winding of the curve
between $e$ and $-\infty$. This winding is well defined up to an
additive constant, and we set it to be equal to $0$ for edges of the
bottom side which point inside the domain. It is easy to see that $F$
is the limit of observables in finite boxes, so that the properties of
fermionic observables in Dobrushin domains carry over to the
infinite-volume case. In particular, the conclusions of the previous
three lemmas apply to it as well.

\section{\texorpdfstring{Proof of Theorem \protect\ref{maintheorem}}{Proof of Theorem 1}}
\label{secproofthm1}

The proof consists of three steps:
\begin{itemize}
\item We first prove using Lemmas~\ref{relationaroundavertex}
and~\ref{argument} that the observable decays exponentially fast when
$p<p_{\mathrm{sd}}$ in a well chosen Dobrushin domain (namely a strip with free
boundary condition on the top and wired boundary condition on the
bottom). Lemma~\ref{boundary} then implies that the probability that a
point on the top of the strip is connected to the bottom decays
exponentially fast in the height of the strip.
\item We derive exponential decay of the connectivity function for the
infinite-volume measure with free boundary conditions from the first
part.
\item Finally, we show that exponential decay implies that the
random-cluster model is subcritical when $p<p_{\mathrm{sd}}$, and that its dual
is supercritical. This last step concludes the proof of
Theorem~\ref{maintheorem} and is classical.
\end{itemize}
In the proof, points are identified with their complex coordinates.

\subsection*{Step 1: Exponential decay in the strip}

Let $p<p_{\mathrm{sd}}$, and consider the random-cluster model on the strip
$\mathcal{S}_\ell$ of height $\ell>0$ with wired boundary condition on
the bottom and free boundary condition on the top. Define $e_{k}$ and
$e_{k+1}$ to be the north-west-pointing sides of the diamonds associated
to the points ${\mathrm i}k$ and ${\mathrm i}(k+1)$, respectively.
Label some of
the edges around these two diamonds as $x$, $x'$, $x''$, $y$ and
$y'$ as shown in Figure~\ref{figeventsurrounding}.

%
%
\begin{figure}[b]

\includegraphics{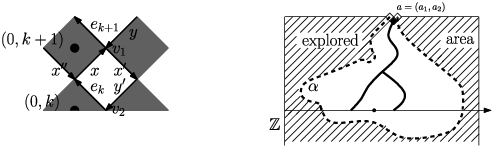}

\caption{Left: the labeling of edges around $e_k$ used in
Step 1. Right: a dual circuit surrounding an open path in
the box $[-a_2,a_2]^2$. Conditioning on to the most exterior such
circuit gives no information on the state of the edges inside it.}
\label{figeventsurrounding}
\end{figure}

Lemmas~\ref{relationaroundavertex} and~\ref{argument} have a very
important consequence: around a vertex~$v$, the value of the observable
on one edge can be expressed in terms of its values on \textit{only two
other edges}. This can be done by seeing the relation given by
Lemma~\ref{relationaroundavertex} as a linear relation between four
vectors in the plane $\R^2$, and applying an orthogonal projection to a
line orthogonal to one of them (which can be chosen using
Lemma~\ref{argument}). One then gets a linear relation between three
real numbers, but using Lemma~\ref{argument} ``in reverse'' shows that
this is enough to determine any of the corresponding three (complex)
values of the observable given the other two.

For instance, we can project (\ref{relvertex}) around $v_1$
orthogonally to $F(y)$, so that we obtain a relation between projections
of $F(x)$, $F(x')$ and $F(e_{k+1})$. Moreover, we know the complex
argument (modulo $\pi$) of $F$ for each edge so that the relation
between projections can be written as a relation between $F(x)$, $F(x')$
and $F(e_{k+1})$ themselves. This leads to
%
%
\begin{equation} \label{eq1}
{\mathrm e}^{-{\mathrm i}\pi/4}F(x) = \cos(\pi/4-\alpha)F(e_{k+1})-\cos
(\pi/4+\alpha){\mathrm e}^{-{\mathrm i}\pi/2}F(x').
\end{equation}
Applying the same reasoning around $v_2$, we obtain
%
%
\begin{equation} \label{eq2}
{\mathrm e}^{-{\mathrm i}\pi/4}F(x) = \cos(\pi/4+\alpha)F(e_k)-\cos
(\pi/4-\alpha){\mathrm e}^{-{\mathrm i}\pi/2}F(x'').
\end{equation}
The translation invariance of $\phi_{\mathcal{S}_\ell}
^{\infty,-\infty}$ implies
%
%
\begin{equation}\label{eq3}
F(x')=F(x'').
\end{equation}
Moreover, symmetry with respect to the imaginary axis implies that
%
%
\begin{equation}
\label{eq4}
F(x)={\mathrm e}^{{\mathrm i}\pi/4}\overline{F(x')}={\mathrm
e}^{-{\mathrm i}\pi/4}F(x').
\end{equation}
Indeed, if, for a configuration $\omega$, $x$ belongs to $\gamma$,
and the
winding is equal to~$W$, in the reflected configuration $\omega'$, $x'$
belongs to $\gamma(\omega')$ and the winding is equal to $\pi/2-W$.

Plugging (\ref{eq3}) and (\ref{eq4}) into (\ref{eq1}) and
(\ref{eq2}), we obtain
\begin{eqnarray*}
F(e_{k+1})&=&{\mathrm e}^{-{\mathrm i}\pi/4}\frac{1+\cos
(\pi/4+\alpha)}{\cos(\pi/4-\alpha)}F(x)\\
&=&\frac{[1+\cos
(\pi/4+\alpha)]\cos(\pi/4+\alpha)}{[1+\cos(\pi/4-\alpha)]\cos
(\pi/4-\alpha)}F(e_k).
\end{eqnarray*}
Remember that $\alpha(p)>0$ since $p<p_{\mathrm{sd}}$,
so that the multiplicative constant is less than 1. Using
Lemma~\ref{boundary} and the previous equality inductively, we find that
there exists $c_1=c_1(p)<1$ such that, for every $\ell> 0$,
\[
\phi_{\mathcal{S}_{\ell}} ^{\infty,-\infty}[{\mathrm i}\ell
\leftrightarrow
\mathbb{Z}] = |F(e_{\ell})| =c_1^\ell|F(e_{1})|\leq c_1^\ell,
\]
where the last inequality is due to the fact that the observable has
complex modulus less than $1$.

\subsection*{Step 2: Exponential decay for $\phi^0_p$ when $p<p_{\mathrm{sd}}$}

Fix again $p<p_{\mathrm{sd}}$. Let $N\in\mathbb{N}$, and recall that
$\phi^0_{p,N}:=\phi^0_{p,2,[-N,N]^2}$ converges to the infinite-volume
measure with free boundary conditions $\phi^0_p$ when $N$ goes to
infinity.

Consider a configuration in the box $[-N,N]^2$, and let $A_{\max}$
be the site of the cluster of the origin which maximizes the
$\ell^\infty$-norm $\max\{|x_1|,|x_2|\}$ (it could be equal to $N$). If
there is more than one such site, we consider the greatest one in
lexicographical order. Assume that $A_{\max}$ equals $a=a_1+{\mathrm
i}a_2$ with $a_2\geq|a_1|$ (the other cases can be treated the same
way by symmetry, using the rotational invariance of the lattice).

By definition, if $A_{\max}$ equals $a$, $a$ is connected to $0$ in
$[-a_2,a_2]^2$. In addition to this, because of our choice of the free\vadjust{\goodbreak}
boundary condition, there exists a dual circuit starting from $a +
\mathrm i/2$ in the dual of $[-a_2,a_2]^2$ (which is the same as
$\mathbb L^* \cap[-a_2-1/2, a_2+1/2]^2$) and surrounding both $a$ and
$0$. Let $\Gamma$ be the outermost such dual circuit: we get
%
%
\begin{equation}
\label{eqsumongamma}
\phi_{p,N}^0(A_{\max}=a) = \sum_{\gamma}
\phi_{p,N}^0(a\leftrightarrow
0|\Gamma=\gamma)\phi_{p,N}^0(\Gamma=\gamma),
\end{equation}
where the sum is over contours $\gamma$ in the dual of $[-a_2,a_2]^2$
that surround both $a$ and $0$.

The event $\{\Gamma=\gamma\}$ is measurable in terms of edges outside or
on $\gamma$. In addition, conditioning on this event implies that the
edges of $\gamma$ are dual-open. Therefore, from the domain Markov
property, the conditional distribution of the configuration inside
$\gamma$ is a random-cluster model with free boundary condition.
Comparison between boundary conditions implies that the probability of
$\{a\leftrightarrow0\}$ conditionally on $\{\Gamma=\gamma\}$ is smaller
than the probability of $\{a\leftrightarrow0\}$ in the strip
$\mathcal{S}_{a_2}$ with free boundary condition on the top and wired
boundary condition on the bottom. Hence, for any such $\gamma$, we get
\[
\phi_{p,N}^0(a\leftrightarrow0|\Gamma=\gamma) \leq
\phi_{\mathcal{S}_{a_2}}^{\infty,-\infty}(a\leftrightarrow0)=
\phi_{\mathcal{S}_{a_2}}^{\infty,-\infty}(a\leftrightarrow
\mathbb{Z})\leq c_1^{a_2}= c_1^{|a|/2}
\]
(observe that for the second
measure, $\mathbb{Z}$ is wired, so that $\{a\leftrightarrow0\}$ and
$\{a\leftrightarrow\mathbb{Z}\}$ have the same probability). Plugging
this into (\ref{eqsumongamma}), we obtain
\[
\phi_{p,N}^0(A_{\max}=a) \leq
\sum_{\gamma}c_1^{|a|/2} \phi_{p,N}^0(\Gamma=\gamma)\le
c_1^{|a|/2}.
\]

Fix $n\leq N$. Since $c_1<1$, we deduce from the previous inequality
that there exist two constants $0<c_2,C_2<\infty$ such that
\begin{eqnarray*}
\phi_{p,N}^0(0\leftrightarrow\mathbb{Z}^2\setminus[-n,n]^2)&\leq&
\sum_{a\in[-N,N]^2\setminus[-n,n]^2}\phi_{p,N}^0(A_{\max}=a)\\
&\leq&
\sum_{a\notin[-n,n]^2} c_1^{|a|/2}\leq C_2{\mathrm e}^{-c_2n}.
\end{eqnarray*}
Since the estimate is uniform in $N$, we deduce that
%
%
\begin{equation}
\label{expon}
\phi_{p}^0(0\leftrightarrow
\mathbb{Z}^2\setminus[-n,n]^2)\leq C_2{\mathrm e}^{-c_2n}.
\end{equation}

\subsection*{Step 3: Exploiting exponential decay}

The inequality $p_{c}\geq p_{\mathrm{sd}}$ follows from (\ref{expon}) since
exponential decay prevents the existence of an infinite cluster for
$\phi^0_p$ when $p<p_{\mathrm{sd}}$.

In order to prove that $p_c\leq p_{\mathrm{sd}}$, we use the following standard
reasoning. Let $A_n$ be the event that the point $(n,0)$ is in an open
circuit which surrounds the origin. Notice that this event is included
in the event that the point $(n,0)$ is in a cluster of radius larger
than $n$. For $p<p_{\mathrm{sd}}$, (\ref{expon}) implies that the probability of
$A_n$ decays exponentially fast. The Borel--Cantelli lemma shows that
there is almost surely only a finite number of values of $n$ such that
$A_n$ occurs. In other words, there is only a finite number of open
circuits surrounding the origin, which enforces the existence of an
infinite dual cluster. It means that the dual model is supercritical
whenever $p<p_{\mathrm{sd}}$. Equivalently, the primal model is supercritical
whenever $p>p_{\mathrm{sd}}$, which implies $p_c\leq p_{\mathrm{sd}}$.

\section{\texorpdfstring{Proof of Theorem \protect\ref{maintheorem2}}{Proof of Theorem 2}}
\label{secproofthm2}

In this section, we compute the correlation length in all directions. In
\cite{Messikh}, Messikh noticed that this correlation length was
connected to large deviations for random walks and asked whether there
exists a direct proof of the correspondence. Indeed, large deviations
results are easy to obtain for random walks, so that one could deduce
Theorem~\ref{maintheorem2} easily. In the following, we exhibit what we
believe to be the first direct proof of this result.

An equivalent way to deal with large deviations of the simple random
walk is to study the \textit{massive Green function} $G_m$, defined in the
bulk as
\[
G_m(x,y) := \mathbb E^x \biggl[ \sum_{n \geq0} m^n \mathbh
1_{X_n=y} \biggr],
\]
where $\mathbb E^x$ is the law of a simple random
walk starting at $x$.

The correlation length of the two-dimensional Ising model is the same as
the correlation length for its random-cluster representation so that we
will state the result in terms of the random-cluster. We use the
parameters $p$ and $\alpha=\alpha(p)$ without revealing the connection
with $\beta$ in the notation.
%
%
\begin{proposition}
For $p<p_{\mathrm{sd}}$ and any $a\in\mathbb{L}$,
%
%
\begin{equation}
\label{messikh}
-\lim_{n\rightarrow\infty}\frac1n\log\phi_{p}^0(0\leftrightarrow
na)= -\lim_{n\rightarrow\infty} \frac1n \log G_m(0,na),
\end{equation}
where $m = \cos[2\alpha(p)]$---the value of $\alpha(p)$ is given by
(\ref{eqalphafromp}).
\end{proposition}

In \cite{Messikh}, the statement involves Laplace transforms, but we can
translate it into the previous terms. Moreover, the mass is expressed in
terms of $\beta$, but it is elementary to compute it in terms of
$\alpha$. Theorem~\ref{maintheorem2} follows from this proposition by
first relating the two-point functions of the Ising and $q=2$
random-cluster models, as was mentioned earlier, and then deriving the
asymptotics of the massive Green function explicitly---the details can
be found, for instance, in the proof of Proposition 8 in \cite{Messikh}.

Before delving into the actual proof, here is a short outline of the
strategy we employ. We have already seen exponential decay in the strip,
which was an essentially one-dimensional computation; we want to refine
it into a two-dimensional version for correlations between two points
$0$ and $a$ in the bulk, and once again we use the observable to
estimate them. The basic step, namely obtaining local linear relations
between the values of the observable, is the same, although it is
complicated by the lack of translation invariance. The point is that the
observable is massive harmonic when $p\neq p_{\mathrm{sd}}$ (see
Lemma~\ref{massivelaplacian} below). Since $G_m(\cdot,\cdot)$ is
massive harmonic in both variables away from the diagonal $x=y$, it is
possible to compare both quantities.

The main problem is that we are interested in correlations in the
bulk. The observable can be defined directly in the bulk (see below), but
it provides only a lower bound on the correlations. In order to obtain
an upper bound, we have to introduce an ``artificial'' domain [that will
be $T(a)$ below], which needs two features: the observable in it can be
well estimated, and at the same time correlations inside it have
comparable probabilities to correlations in the bulk. For the second
one, it is equivalent to impose that the Wulff shape centered at $0$, and
having $a$ on its boundary is contained in the domain in the
neighborhood of $a$; from convexity, it is then natural to construct
$T(a)$ as the whole plane minus two wedges, one with vertex at $0$ and
the other with vertex at $a$.

The proof is rather technical since we need to deal with the behavior of
the observable on the boundary of the domains. This was also an issue in
Smirnov's proof. At criticality, the difficulty was overcome by working
with the discrete primitive $H$ of $F^2$. Unfortunately, there is no
nice equivalent of $H$ to work with away from criticality. The solution
is to use a representation of $F$ in terms of a massive random
walk. This representation extends to the boundary and allows to control
the behavior of $F$ everywhere.
\begin{pf*}{Proof of Theorem \protect\ref{maintheorem2}}
Let $p<p_{\mathrm{sd}}$. Without loss of generality, we can consider $a =
(a_1,a_2) \in\mathbb{L}$ satisfying $a_2\geq a_1\geq0$. In the
proof, we identify a site $u$ of $\mathbb{L}$ with the unique side
$e_u$ of the associated black diamond which points north-west. In
other words $F(u)$ and $\{u\in\gamma\}$ should be understood as
$F(e_u)$ and $\{e_u\in\gamma\}$---notice that this differs from the
notation used in~\cite{Smirnov}.\vspace*{8pt}

\textit{The lower bound.}
Consider the observable $F$ in the bulk defined as follows: for every
edge $e$ not equal to $e_0$,
%
%
\begin{equation}
\label{defFbulk}
F(e) := \phi_p^0 \bigl({\mathrm e}^{({\mathrm i}/{2})
\mathrm{W}_{\gamma}(e,e_0)} \mathbh{1}_{e\in\gamma}\bigr),
\end{equation}
where $\gamma$ is the unique loop passing through $e_0$. Note that
this definition is justified by the fact that $p$ is subcritical, and
that it immediately implies that
%
%
\begin{equation}
\label{upper1}
\phi_{p}^0(0\leftrightarrow a) \ge|F(a)|.
\end{equation}
We mention that $F$ is not well defined at $e_0$. Indeed, $e_0$ can be
thought of as the start of the loop $\gamma$ or its end. In other
words, $F$ is multi-valued at $e_0$, with value 1 or $-$1.

Lemma~\ref{relationaroundavertex} can be extended to this context
following a very similar proof, but taking into account that $F$ is
multi-valued\vadjust{\goodbreak} at $e_0$. More precisely, let $e_0=xy$. Around any vertex
$v \notin\{x,y\}$ the relation in
Lemma~\ref{relationaroundavertex} still holds; besides,
\[
\cases{
F(\mathit{SE})+1 = {\mathrm e}^{{\mathrm i}\alpha(p)} [F(\mathit{SW})+F(\mathit{NE})],
&\quad
if $v = y$,\vspace*{2pt}\cr
F(\mathit{SW})+F(\mathit{NE}) = {\mathrm e}^{{\mathrm i}\alpha(p)} [-1+F(\mathit{SE})],
&\quad
if $v = x$,}
\]
where the $\mathit{NE}$ (resp., $\mathit{SE}$, $\mathit{SW}$) is the edge at $v$
pointing to the north-east (resp., south-east, south-west). In other
words, the statement of Lemma~\ref{relationaroundavertex} still
formally holds if we choose the convention that $F(e_0)=1$ when
considering the relation around $x$, and $F(e_0)=-1$ when considering
the relation around $y$.

One can see that Lemma~\ref{argument} is still valid. In fact, the two
lemmas imply that $F$ is massive harmonic:
%
%
\begin{lemma}
\label{massivelaplacian}
Let $p<p_{\mathrm{sd}}$ and consider the observable $F$ in the bulk. For
any site $X$ not equal to 0, we have
\[
\Delta_\alpha F(X) := \frac{\cos2\alpha}{4}[F(W) + F(S) + F(E) +
F(N)]-F(X) = 0,
\]
where $W$, $S$, $E$ and $N$ are the four neighbors of $X$.
\end{lemma}
\begin{pf}
Consider a site $X$ inside the domain and recall that we identify
$X$ with the corresponding edge of the medial lattice pointing
north-west. Index the edges around $X$ in the same way as in case 1
of Figure~\ref{figlemma}. By considering the six equations
%
%
\begin{figure}

\includegraphics{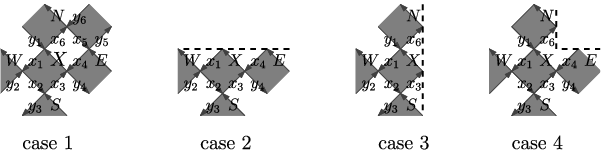}

\caption{Indexation of the edges around vertices in the different cases.}
\label{figlemma}
\end{figure}
corresponding to vertices that end one of the edges $x_1, \ldots,
x_6$ (being careful to identify the edges $A$, $B$, $C$ and $D$
correctly for each of the vertices), we obtain the following linear
system:
\[
\cases{
F(X) + F(y_1) = {\mathrm e}^{{\mathrm i}\alpha}[F(x_1)+F(x_6)],
\vspace*{1pt}\cr
F(y_2) + F(x_1) = {\mathrm e}^{{\mathrm i}\alpha}[F(x_2)+F(W)],
\vspace*{1pt}\cr
F(S)+F(x_2)={\mathrm e}^{{\mathrm i}\alpha}[F(y_3)+F(x_3)],\vspace*{1pt}\cr
F(x_3)+ F(x_4)={\mathrm e}^{{\mathrm i}\alpha}[F(y_4)+F(X)], \vspace*{1pt}\cr
F(E)+F(x_5)={\mathrm e}^{{\mathrm i}\alpha}[F(x_4)+F(y_5)], \vspace*{1pt}\cr
F(x_6)+F(y_6)={\mathrm e}^{{\mathrm i}\alpha}[F(x_5)+F(N)].}
\]

Recall\vspace*{1pt} that by definition, $F(X)$ is real. For an edge $e$, denote
by $f(e)$ the projection of $F(e)$ on the line directed by its
argument ($\mathbb{R}$, ${\mathrm e}^{{\mathrm i}\pi/4}\mathbb{R}$,
$i\mathbb{R}$ and ${\mathrm e}^{-{\mathrm i}\pi/4}\mathbb{R}$). By
projecting orthogonally to the $F(y_i)$, $i=1,\ldots,6$, the system
becomes
\[
\cases{
f(X)=\cos(\pi/4+\alpha)f(x_1)+\cos
(\pi/4-\alpha)f(x_6), & (1)\vspace*{1pt}\cr
f(x_1)=\cos(\pi/4+\alpha)f(x_2)+\cos(\pi/4-\alpha)
f(W),& (2)\vspace*{1pt}\cr
f(x_3)=\cos(\pi/4-\alpha)f(S)+-\cos
(\pi/4+\alpha)f(x_2),& (3)\vspace*{1pt}\cr
f(X)=\cos(\pi/4+\alpha)f(x_3)+\cos(\pi/4-\alpha)
f(x_4),& (4)\vspace*{1pt}\cr
f(x_4)=\cos(\pi/4+\alpha)f(E)+\cos
(\pi/4-\alpha)f(x_5),& (5)\vspace*{1pt}\cr
f(x_6)=-\cos(\pi/4-\alpha)f(x_5)+\cos
(\pi/4+\alpha)f(N).& (6)}
\]
By adding $(2)$ to $(3)$, $(5)$ to $(6)$ and $(1)$ to $(4)$, we
find
\[
\cases{
f(x_3)+f(x_1)=\cos(\pi/4-\alpha)[f(W)+f(S)],& (7)\vspace*{1pt}\cr
f(x_6)+f(x_4)=\cos(\pi/4+\alpha)[f(E)+f(N)],& (8)\vspace*{1pt}\cr
2f(X)=\cos(\pi/4+\alpha)[f(x_3)+f(x_1)]+\cos
(\pi/4-\alpha)[f(x_6)+f(x_4)].& (9)}
\]
Plugging $(7)$ and $(8)$ into $(9)$, we obtain
\[
2f(X)=\cos(\pi/4+\alpha)\cos(\pi/4-\alpha)[f(W)+f(S)+f(E)+f(N)].
\]
The edges $X, \ldots, N$ are pointing in the same direction so the
previous equality becomes an equality with $F$ in place of $f$ (use
Lemma~\ref{argument}). A simple trigonometric identity then leads to
the claim.
\end{pf}

Define the Markov process with generator $\Delta_\alpha$, which one
can see either as a branching process or as the random walk of a
massive particle. We choose the latter interpretation and write this
process $(X_n,m_n)$ where $X_n$ is a random walk with jump
probabilities defined in terms of $\Delta_\alpha$---the
proportionality between jump probabilities is the same as the
proportionality between coefficients---and $m_n$ is the mass
associated to this random walk. The law of the random walk starting
at $x$ is denoted $\mathbb{P}^x$. Note that the mass of the walk
decays by a factor $\cos2\alpha$ at each step.

Denote by $\tau$ the hitting time of $0$. The last lemma translates
into the following formula for any $a$ and any $t$:
%
%
\begin{equation}
\label{timet}
F(a)=\mathbb{E}^a[F(X_{t\wedge\tau})m_{t\wedge\tau}].
\end{equation}
The sequence $(F(X_t) m_t)_{t \leq\tau}$ is obviously uniformly
integrable, so that (\ref{timet}) can be improved to
%
%
\begin{equation}
\label{upper2}
F(a)=\mathbb{E}^a[F(X_\tau)m_\tau].
\end{equation}
Equations (\ref{upper1}), (\ref{upper2}) together with
Lemma~\ref{upper3} below give
\[
\phi_{p}^0(0\leftrightarrow a)\geq
\frac{c}{|a|}G_{\cos2\alpha}(0,a),
\]
which implies the lower bound.
%
%
\begin{lemma}
\label{upper3}
There exists $c>0$ such that, for every $a$ in the upper-right
quadrant,
\[
|\mathbb{E}^a[F(X_{\tau})m_{\tau}]|\geq
\frac{c}{|a|}G_{\cos2\alpha}(0,a).
\]
\end{lemma}
\begin{pf}
Recall that $F(X_{\tau})$ is equal to 1 or $-1$ depending on the last
step the walk takes before reaching 0. Let us rewrite
$\mathbb{E}^a[F(X_{\tau})m_{\tau}]$ as
\[
\mathbb{E}^a\bigl[m^\tau1_{\{X_{\tau-1}=W \ \mathrm{or}\ S\}}\bigr] -
\mathbb{E}^a\bigl[m^\tau1_{\{X_{\tau-1}=N\ \mathrm{or}\ E\}}\bigr].
\]
Now, let $\Delta_\alpha$ be the line $y=-x$, and let $T$ be the time
of the last visit of $\Delta_\alpha$ by the walk before time $\tau$
(set $T= \infty$ if it does not exist). On the event that
$X_{\tau-1}=W$ or $S$, this time is finite, and reflecting the part
of the path between $T$ and $\tau$ across $\Delta_\alpha$ produces a
path from $a$ to $0$ with $X_{\tau-1}=E$ or~$N$. This transformation
is one-to-one, so summing over all paths, we obtain
\begin{eqnarray*}
&&
\mathbb{E}^a\bigl[m^\tau1_{\{X_{\tau-1}=W \ \mathrm{or}\ S\}}\bigr] -
\mathbb{E}^a\bigl[m^\tau1_{\{X_{\tau-1}=N\ \mathrm{or}\ E\}}\bigr]\\
&&\qquad = -
\mathbb{E}^a\bigl[m^\tau1_{\{X_{\tau-1}=N\ \mathrm{or}\ E\}}
1_{\{T=\infty\}}\bigr],
\end{eqnarray*}
which in turn is equal to
$-\mathbb{E}^a[m^\tau1_{\{T=\infty\}}]$. General arguments of large
deviation theory imply that $\mathbb{E}^a[m^\tau1_{\{T=\infty\}}]\ge
\frac{c}{|a|}G_{\cos2\alpha}(0,a)$ for some universal constant $c$.
\end{pf}

\textit{The upper bound.}
Assume that $0$ is connected to $a$ in the bulk. We first show how to
reduce the problem to estimations of correlations for points on the
boundary of a domain.

For every $u=u_1+{\mathrm i}u_2$ and $v=v_1+{\mathrm i}v_2$ two sites of
$\mathbb{L}$, write $u\prec v$ if $u_1< v_1$ and $u_2< v_2$. This
relation is a partial ordering of $\mathbb{L}$. We consider the
following sets:
\[
\mathbb{L}^+(u) = \{x\in\mathbb{L}\dvtx u\prec x\} \quad\mbox{and}\quad
\mathbb{L}^- = \{x\in\mathbb{L}\dvtx x\prec0\}
\]
and
\[
T(u)=\mathbb{L} \setminus\bigl(\mathbb{L}^+(u)\cup\mathbb{L}^-\bigr).
\]
In
the following, $L^+(u)$ and $L^-$ will denote the interior boundaries
of $T(u)$ near $\mathbb{L}^+(u)$ and $\mathbb{L}^-$, respectively; see
Figure~\ref{figlaplacian}. The measure with wired boundary conditions
on $\mathbb{L}^-$ and free boundary conditions on $\mathbb{L}^+(u)$ is
denoted $\phi_{T(u)}$.

%
%
\begin{figure}

\includegraphics{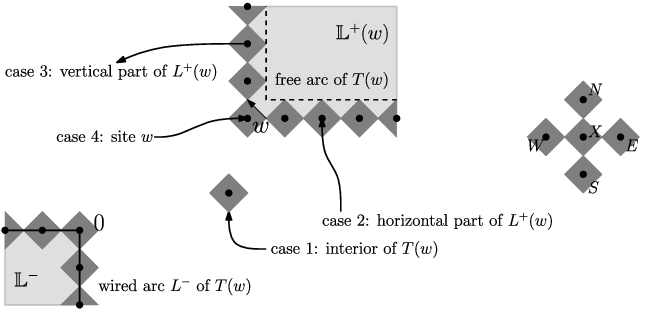}

\caption{The set $T(w)$. The different cases listed in
the definition of the Laplacian are pictured.}
\label{figlaplacian}
\end{figure}

Assume that $a$ is connected to 0 in the bulk. By conditioning on $w$
which maximizes the partial $\succ$-ordering in the cluster of $0$ (it
is the same reasoning as in Section~\ref{secproofthm1}), we obtain the
following:
%
%
\begin{equation}
\label{firstinequality}
\phi_p^0(a\leftrightarrow0)\leq\sum_{w\succ
a}\phi_{T(w)}(w\leftrightarrow\mathbb{L}^-)\leq
C_3|a|\max_{w\succ a,|w|\leq c_3|a|}\phi_{T(w)}(w\leftrightarrow
\mathbb{L}^-)\hspace*{-35pt}
\end{equation}
for $c_3,C_3$ large enough. The existence of $c_3$ is given by the
fact that the two-point function decays exponentially fast: a
priori estimates on the correlation length show that the maximum
above cannot be reached at any $w$ which is much further away from the
origin than $a$, and even that the sum of the corresponding
probabilities is actually of a smaller order than the remaining terms.
Summarizing, it is sufficient to estimate the probability of the
right-hand side of (\ref{firstinequality}).

Observe that $w$ is on the free arc of $T(w)$, so that, harnessing
Lemma~\ref{boundary}, we find
%
%
\begin{equation}
\label{y}
\phi_{T(w)}(w\leftrightarrow L^-)=|F(w)|,
\end{equation}
where $F$ is the observable in the Dobrushin domain $T(w)$ (the
winding is fixed in such a way that it equals 0 at $e_w$). Now,
similarly to Lemma~\ref{massivelaplacian}, $F$~satisfies local
relations in the domain $T(w)$:
%
%
\begin{lemma}
\label{massivelaplacian2}
The observable $F$ satisfies $\Delta_\alpha F=0$ for every site
\textit{not on the wired arc}, where the massive Laplacian
$\Delta_\alpha$ on $T(w)$ is defined by the following relations: for
all $g\dvtx T(w)\mapsto\mathbb{R}$, $(g+\Delta_\alpha g)(X)$ is equal
to
\[
\frac{\cos2\alpha}{4}[g(W)+g(S)+g(E)+g(N)]
\]
inside the domain;
\[
\frac{\cos2\alpha}{2(1+\cos
(\pi/4-\alpha))}[g(W)+g(S)]+\frac{\cos(\pi/4+\alpha)}{1+\cos
(\pi/4-\alpha)}g(E)
\]
on the horizontal part of $L^+(w)$;
\[
\frac{\cos2\alpha}{2(1+\cos
(\pi/4-\alpha))}[g(W)+g(S)]+\frac{\cos(\pi/4+\alpha)}{1+\cos
(\pi/4-\alpha)}g(N)
\]
on the vertical part of  $L^+(w)$;
\[
\frac{\cos2\alpha}{4}[g(W)+g(S)]+\frac{\cos
(\pi/4-\alpha)}{2}[g(E)+g(N)] \qquad \mbox{at } w
\]
with $N$, $E$, $S$ and $W$ being the four neighbors
of $X$.
\end{lemma}
\begin{pf}
When the site is inside the domain, the proof is the same as in
Lem\-ma~\ref{massivelaplacian}. For boundary sites, a similar
computation can be done. For instance, consider case 2 in
Figure~\ref{figlemma}. Equations (3) and (7) in the proof of
Lemma~\ref{massivelaplacian} are preserved. Furthermore,
Lemma~\ref{boundary} implies that
\[
f(X)=f(x_1)=\phi_{T(w)}(X
\leftrightarrow L^-)
\]
and similarly $f(x_4)=f(E)$ (where $f$ is
still as defined in the proof of Lem\-ma~\ref{massivelaplacian}).
Plugging all these equations together, we obtain the second
equality. The other cases are handled similarly.
\end{pf}

Now, we aim to use a representation with massive random walks similar
to the proof of the lower bound. One technical point is the fact that
the mass at $w$ is larger than 1. This could a priori prevent
$(F(X_t)m_t)_{t}$ from being uniformly integrable. Therefore, we need
to deal with the behavior at $w$ separately. Denote by $\tau_1$ the
hitting time (for $t>0$) of $w$, and by $\tau$ the hitting time of
$L^-$. Since the masses are smaller than 1, except at~$w$,
$(F(X_t)m_t)_{t\le\tau\wedge\tau_1}$ is uniformly integrable and we
can apply the stopping theorem to obtain
\[
F(w)=\mathbb{E}^w[F(X_{\tau\wedge\tau_1}) m_{\tau\wedge
\tau_1}]=\mathbb{E}^w[F(X_{\tau_1})
m_{\tau_1}\mathbh{1}_{\tau_1<\tau}]+\mathbb{E}^w[F(X_\tau)
m_\tau\mathbh{1}_{\tau<\tau_1}].
\]
Since $X_{\tau_1}=w$, the previous formula can be rewritten as
%
%
\begin{equation}
\label{lower1}
F(w) = \frac{\mathbb{E}^w [F(X_\tau)
m_\tau\mathbh{1}_{\tau<\tau_1}]}
{1-\mathbb{E}^w(m_{\tau_1}\mathbh{1}_{\tau_1<\tau})}.
\end{equation}

When $w$ goes to infinity in a prescribed direction,
$[1-\mathbb{E}^w(m_{\tau_1}\mathbh{1}_{\tau_1<\tau})]$ converges to
the analytic function $h\dvtx[0,1]\rightarrow\mathbb{R}, p\mapsto
1-\mathbb{E}^w(m_{\tau_1})$ (since the function is
translation-invariant). The function $h$ is not equal to 0 when $p=0$,
implying that it is equal to 0 for a discrete set $\mathcal{P}$ of
points. In particular, for $p\notin\mathcal{P}$, the first term in
the right-hand side stays bounded when $w$ goes to infinity. Denoted
by $C_4=C_4(p)$ such a bound. Recalling that $|F|\leq1$ and that the
mass is smaller than~1 except at $w$, (\ref{lower1}) becomes
%
%
\begin{eqnarray}
|F(w)|&\leq& C_4|\mathbb{E}^w[F(X_\tau)
m_\tau\mathbh{1}_{\tau<\tau_1}]|\leq
\mathbb{E}^w[m_\tau\mathbh{1}_{\tau<\tau_1}]\\
\label{secondinequality}
&\leq& C_4\sum_{w\prec x}\mathbb{E}^x\bigl[(\cos
2\alpha)^\tau\mathbh{1}_{\tau<\tau_1}1_{\{
(X_t)\ \mathrm{avoids}\ L^+(w) \}}\bigr]\nonumber\\[-8pt]\\[-8pt]
&\leq& C_4\sum_{w\prec x}G_{\cos
2\alpha}(0,x),\nonumber
\end{eqnarray}
where the last inequality is due to the fact that we release the
condition on avoiding $L^+(w)$.\vadjust{\goodbreak}

Finally, it only remains to bound the right-hand side. From
(\ref{secondinequality}), we deduce
%
%
\begin{equation}
\label{thirdequation}
|F(w)|\leq C_5|w|G_{\cos
2\alpha}(0,w),
\end{equation}
where the existence of $C_5$ is due to the exponential decay of
$G_{\cos2\alpha}(\cdot,\cdot)$ and the fact that $G_{\cos
2\alpha}(0,x)\leq G_{\cos2\alpha}(0,w)$ whenever $w\prec x$. We
deduce from~(\ref{firstinequality}), (\ref{y}) and
(\ref{thirdequation}) that
%
%
\begin{equation}\qquad
\phi_p(0\leftrightarrow a)\leq C_3C_5|a|^2\max_{w\succ
a,|w|_{\infty}\leq c_5|a|_{\infty}} G_m(0,w)\leq C_6|a|^2G_m(0,a).
\end{equation}
Taking the logarithm, we obtain the claim for all $p<p_{\mathrm{sd}}$ not in
the discrete set $\mathcal P$. The result follows for every $p$ using
the fact that the correlation length is increasing in $p$.
\end{pf*}

\section*{Acknowledgments}

This work has been accomplished during the stay of the first author in
Geneva. The authors would like to thank the anonymous referee for useful
comments on a previous version of this paper. The second author
expresses his gratitude to S. Smirnov for his constant support during
his Ph.D.


%

\printaddresses

\end{document}